\documentclass[12pt]{amsart}
\usepackage{fullpage,url,amssymb,amsmath,graphicx}

\usepackage{amsmath, amsthm, amssymb}
\usepackage{url}
\usepackage{graphicx}
\usepackage {amsmath}
\usepackage{amsthm}
\usepackage {amssymb}

\usepackage{amsmath}
\usepackage {amssymb}
\usepackage {graphicx,enumerate,booktabs}
\usepackage {color, tikz}
\usepackage[all]{xy}

\newtheorem{definition}{Definition}
\newtheorem{lemma}{Lemma}[section]
\newtheorem{theorem}[lemma]{Theorem}

\newtheorem{cor}[lemma]{Corollary}

\newtheorem{claim*}{Claim}
\newtheorem{remark}[lemma]{Remark}

\newtheorem{example}[lemma]{Example}

\numberwithin{equation}{section}

\newcommand{\Hom}{\operatorname{Hom}}
\newcommand{\C}{\operatorname{Cone}}
\newcommand{\Sym}{\operatorname{Sym}}
\newcommand{\AC}{\mathcal{AC}}
\newcommand{\MC}{\mathcal{MC}}

\newcommand{\pp}{\mathbb{P}}

\newcommand{\rr}{\mathbb{R}}

\usepackage{tikz, tikz-3dplot}

\definecolor{cof}{RGB}{219,144,71}
\definecolor{pur}{RGB}{186,146,162}
\definecolor{greeo}{RGB}{91,173,69}
\definecolor{greet}{RGB}{52,111,72}

\begin{document}
\title{Linearization functors on real convex sets.}
\date{}
\author{
Mauricio Velasco\\ 
}
\address{
Departamento de matem\'aticas\\
Universidad de los Andes\\
Carrera $1^{\rm ra}\#18A-12$\\ 
Bogot\'a, Colombia
}
\email{mvelasco@uniandes.edu.co}

\begin{abstract}
We prove that linearizing certain families of polynomial optimization problems leads to new functorial operations in real convex sets. We show that these operations can be computed or approximated in ways amenable to efficient computation. These operations are convex analogues of Hom functors, tensor products, symmetric powers, exterior powers and general Schur functors on vector spaces and lead to novel constructions even for polyhedra.
\end{abstract}

\maketitle

\section{Introduction}

Convex polynomial optimization is concerned with the problem of determining the maximum value of a real polynomial function $f$ on a real convex set $C$. In the special case when the polynomial is linear and the convex set is an SDR set  (i.e. a projection of a spectrahedron) this problem can be solved numerically very efficiently, in polynomial time on the length of the description of $C$ (see for instance~\cite{NN}). 

In this context it is natural to ask whether we can {\it linearize} arbitrary polynomial optimization problems, that is, whether we can construct a {\it linear} function $F(f)$ and a convex domain $F(C)$ with $\max_{x\in C} f(x)=\max_{y\in F(C)}F(f)(y)$. The main result of this article is to answer this question affirmatively for various classes of non-linear problems. To do so we introduce several {\it linearization functors} on real convex sets which are the convex analogues of tensor and symmetric powers and more generally Schur functors on vector spaces. These operations give us a procedure to build the functions $F(f)$ as well as the new convex domains $F(C)$. 

The scope of this method depends on whether the sets $F(C)$ admit descriptions amenable to efficient computation.  and one of the main results of this article is the construction of arbitrarily accurate approximation schemes for the sets $F(C)$ via projections of spectrahedra.

As in the category of vector spaces the objects which linearize certain families of maps are most clearly understood in terms of ``universal properties". To this end we introduce the categories of almost-compact cones and of marked cones. We show that our linearization functors can be understood as solutions to universal problems in these categories. Our categorical point of view extends the approach pioneered by Ziegler~\cite[$\S$ 9.4]{Z} and Bogart, Contois, Gubeladze~\cite{BCG} from polyhedral cones to general convex cones. In particular, our results extend the results of~\cite{BCG} on tensor products and Hom functors. In more detail, the contributions and organization of this article can be summarized as follows:
\begin{enumerate}
\item{ In Section $\S$\ref{Categories} we introduce the categories of almost compact cones and of marked cones. We define hom functors, tensor products and symmetric power operations in these categories. In Section~$\S$\ref{Schur functors} we define, for each partition $\lambda$ a Schur functor $\mathbb{S}_{\lambda}$ in the category of marked cones.}
\item{ In Section $\S$\ref{Sec: Functors} we show that the functors in the previous paragraph arise as solutions of linearization problems (Theorem~\ref{Lemma: Univ}) and thus give us a new approach to several nonlinear polynomial optimization problems. We also study the facial structure of the cones obtained by applying linearization functors (see Theorem~\ref{Faces}). As an application we show in Example~\ref{TSP} that natural nonlinear extensions of the traveling salesman problem lead to families of new polytopes satisfying the universality property of Billera and Sarangarajan (i.e. contain faces isomorphic to every $\{0,1\}$ polytope).}
\item{ In Sections $\S$\ref{Sec: Approx} and $\S$\ref{Sec: Computing} we address the question of how to compute linearization functors. This is done in two ways:
\begin{enumerate}
\item{In Section $\S$\ref{Sec: Approx} we introduce a general approximation scheme for polynomial images of compact convex sets which support a measure via projections of spectrahedra. This scheme induces a hierarchy of relaxations which we show converges to the desired set. Since all linearization functors are polynomial this method gives a way to approximate them, with arbitrary precision, via projections of spectrahedra.}
\item{In Section $\S$\ref{Sec: Computing} we focus on computing linearization functors in the case of spectrahedral and SDR cones. The behavior of linearization functors on morphisms allows us, in some cases, (see Theorem~\ref{morp}) to reduce this computation to that of linearization functors applied to the PSD cones  $S_+(V)$ of positive semidefinite quadratic forms on the vector space $V$.}
\end{enumerate}
}
\item{Finally in Section $\S$\ref{Sec: PSD} we study the convex geometry of the cones obtained by applying linearization functors to PSD cones. We show that the cones $S_+(V)\otimes S_+(W)$ and $\Hom(S_+(V),S_+(W))$ have various natural interpretations allowing us to prove that, in general, these sets are not spectrahedra (and the latter are not even basic closed semialgebraic). In particular, this shows that the subcategory of spectrahedral cones, unlike that of polyhedral cones, is not closed under either tensor powers or Hom functors.}
\end{enumerate}
To conclude, we would like to propose the following open problem which stems naturally from the results in this article: Are the tensor powers, hom functors and symmetric powers of SDR cones also SDR cones?

\subsection*{Preliminaries} All vector spaces in this article are over the field of real numbers. If $V$ is a vector space then a cone $C$ in $V$ is a subset $C\subseteq V$ closed under nonnegative linear combinations of its elements. A face $F$ of a convex set $C$ is a cone $F\subseteq C$ such that if $c_1+c_2\in F$ for $c_1,c_2 \in C$ then $c_1,c_2\in F$. A face $F$ is exposed if there exists a linear functional $\phi\in V^*$, called a witness for $F$, such that $\phi(C)\geq 0$ and $F=C\cap {\rm ker}(\phi)$. A convex cone is pointed if the origin is an exposed face and any witness $\phi$ for  $\{0\}$ is called a grading for $C$. A cone is closed if it is a closed subset in the euclidean topology on $V$. By a convex body $P$ in $V$ we mean a full-dimensional compact convex set $P\subseteq V$. For preliminaries on convex sets including duality, polarity, and extreme points the reader should refer to~\cite{Barvinok}. For preliminaries on spectrahedra and SDR sets the reader should refer to Section~$\S$\ref{Sec: subCat} and to~\cite{Blekherman}. By a functor in a category we mean a functor from the category to itself. For preliminaries on Schur functors on vector spaces the reader should refer to~\cite[Section 8.1]{F}. 

\subsection*{Acknowledgements} I would like to thank Gregoriy Blekherman, Tristram Bogart, C\'esar Galindo, Mauricio Junca and Gregory G. Smith for many valuable conversations during the completion of this work.

\section{Categories of real convex sets}\label{Categories}
In this article we study the behavior of functors from multilinear algebra on convex sets in vector spaces. We work in the following two categories:

\begin{definition} Let $\AC$ denote the category of {\bf almost compact cones}. Its objects are pairs $(C,V)$ where $V$ is a finite-dimensional real vector space and $C$ is a pointed, closed and full-dimensional convex cone in $V$. We denote such a pair $(C,V)$ by its first component $C$ and define $LC:=V$. The morphisms between objects $C_1$ and $C_2$, denoted $\Hom_\AC(C_1,C_2)$ are the linear maps $f:LC_1\rightarrow LC_2$ satisfying $f(C_1)\subseteq C_2.$ 
\end{definition}
\begin{definition} Let $\MC$ denote the category of {\bf marked cones}. Its objects are triples $(C,g,s)$ where $C\in \AC$, $g:C\rightarrow \rr_+$ is a morphism and $s:\rr_+\rightarrow C$ is a section of $g$ with $s(1)\in {\rm int}(C)$.  The morphisms between objects $(C_1,g_1,s_1)$ and $(C_2,g_2,s_2)$ are the $f\in \Hom_\AC(C_1,C_2)$ such that $g_2\circ f=g_1$ and $f\circ s_1=s_2$. We will denote triples $(C,g,s)$ by their first component $C$ and denote $g$ and $s$ by $g_C$ and $s_C$ respectively.
\end{definition}
If $E$ is a real vector space and $P\subseteq E$ is a convex body then the cone $C_1={\rm Cone}\{(p,1): p\in P\}\subseteq E\times \rr$ is an almost compact cone and every such cone is the cone over some compact base. If moreover $P$ contains the origin in its interior then $C_1$ can be endowed with a grading $g:C_1\rightarrow \rr_+$ given by projection onto the last component and with a section $s(\alpha)=\alpha(0,1)$ so that $(C_1,g,s)$ is a marked cone. We can recover $P\subseteq E$ from $(C,g,s)$ by letting $E$ be the vector space obtained by making $s(1)$ the origin of the affine space $g^{-1}(1)\subseteq LC$ and letting $P:=g^{-1}(1)\cap C$. It is easy to see that via this construction

\begin{lemma} \label{Equiv} The category of marked cones is equivalent to the category of full-dimensional compact convex sets with $0$ in their interior and morphisms given by restrictions of linear maps.
\end{lemma}

If $f\in \Hom_\MC(C_1,C_2)$ then $f(C_1)$ is a closed cone (since it is the cone over the compact convex set $f(g_1^{-1}(1)\cap C_1)$ which is contained in $g_2^{-1}(1)$ and thus does not contain the origin).  Closedness may fail for morphisms in $\AC$ as the following example shows:
\begin{example} Let $P$ be the convex set in the plane defined by $y\geq \pm x+x^2$. Let $C={\rm Cone}\{(p,1):p\in P\}\subseteq \rr^3$ and let $D\subseteq \rr^2$ be the cone generated by $(\pm 1,1)$. If $\pi: C\rightarrow D$ is the projection onto the first two components then the set $\pi(C)$ is the interior of $D$ together with the origin and in particular is not a closed cone. \end{example}

The following definitions induce a duality functor and define categorical products in $\AC$ and $\MC$.

\begin{definition} If $C\in \AC$ then its dual $C^*\subseteq LC^*\in \AC$. If $f\in \Hom_\AC(C_1,C_2)$ then the transpose $f^*:LC_2^*\rightarrow LC_1^*$ maps $C_2^*$ to $C_1^*$. If $(C,g,s)\in \MC$ then $(C^*,s^*,g^*)\in \MC$. If $C_1,C_2\in \AC$ then the cartesian product $C_1\times C_2\subseteq LC_1\times LC_2\in \AC$ and if $(C_i,g_i,s_i)\in \MC$ for $i\in\{1,2\}$ then letting $g(c_1,c_2):= \frac{g_1(c_1)+g_2(c_2)}{2}$ we see that $(C_1\times C_2, g,s_1\times s_2)\in \MC.$ 
\end{definition}

\begin{remark} Via the equivalence in Lemma~\ref{Equiv} the duality above recovers the concept of the polar $P^{\circ}$ of a convex set. The resulting product of two compact convex sets $P_1\subseteq E_1$ and $P_2\subseteq E_2$ is the subset $P\subseteq E_1\times E_2\times \rr$ given by $P={\rm Conv}((P_1,0,0), (0,P_2,1)).$
\end{remark}

\subsection{Subcategories of convex sets.}\label{Sec: subCat}
The following distinguished kinds of convex sets play an important role in this paper,

\begin{definition} For a real vector space $W$ the PSD cone $S_+(W)$ is the set of sums of squares in $\Sym^2(W)$. The pairs $(S_+(W),\Sym^2(W))$ are $\AC$ cones. A cone $(C,LC)\in \AC$ is a spectrahedral cone if there exists a real vector space $W$ and an injective linear map $\psi: LC\rightarrow \Sym^2(W^*)$ such that $C=\psi^{-1}(S_+(W^*))$. A cone $(D,LD)\in \AC$ is an SDR (semidefinitely representable) cone if there exist a spectrahedral cone $C$ and a surjective morphism $\pi: C\rightarrow D$.  A cone in $\MC$ is spectrahedral or SDR if all objects and morphisms from the previous paragraph are in $\MC$. 
\end{definition}
\begin{definition} Let $E\cong \rr^n$ be a real vector space. A spectrahedron in $E$ is a set of the form $\{x\in E: A+\sum_{i=1}^n x_iB_i\succeq 0\}$ for some symmetric matrices $A,B_1,\dots B_n$. An SDR set in $E$ is a set of the form $\{x\in E:\exists y\in E'\left(A+\sum_{i=1}^n x_i B_i+\sum_{j=1}^k y_j C_j\succeq 0\right)\}$ for some  real vector space $E'\cong \rr^k$ and some symmetric matrices $A,B_1,\dots B_n,C_1,\dots C_k$.
\end{definition}
Via Lemma~\ref{Equiv} compact spectrahedral (resp. SDR) sets in $E$ determine marked spectrahedral (resp. SDR) cones in $E\times \rr$. Conversely if $C$ is a spectrahedral (resp. SDR) cone in $\MC$ then the sets $C\cap g_C^{-1}(1)$ are spectrahedral (resp. SDR) sets. Spectrahedra and SDR sets play an important role in optimization because the problem of optimizing a linear functional over a spectrahedron (and thus over an SDR set) can be solved in polynomial time on the length of its description  (see~\cite{NN} for precise statements). 

\section{Functors on real convex sets.}\label{Sec: Functors}

In this section we define several linearization operations on convex sets. Our main contribution is to interpret them as solutions to universal problems and to study their facial structure.
\begin{definition} If $C_1,C_2\in \AC$ then $\Hom_\AC(C_1,C_2)$ is a cone in the real vector space $\Hom(LC_1,LC_2)$. Define
$C_1\otimes C_2:=\C\{ c_1\otimes c_2: c_i\in C_i\}\subseteq LC_1\otimes LC_2$ and for any integer $p>0$ define $\Sym^p(C_1)=\{v_1\dots v_p:v_i\in C_1\}\subseteq \Sym^p(LC_1)$. 
\end{definition}

\begin{theorem} If $C_1,C_2\in \AC$  (resp. $\in\MC$) and $p>0$ is an integer then the following statements hold:
\begin{enumerate}
\item{\label{Obj} The cones $\Hom_{\AC}(C_1,C_2)$, $C_1\otimes C_2$ and $\Sym^p(C_1)$ are in $\AC$ (resp. in $\MC$).}
\item{\label{functors} Defining the action on morphisms as in the category of vector spaces then $\Hom_{AC}(C_1,-)$, $\Hom_{\AC}(-, C_2)$, $C_1\otimes -$ and $\Sym^p(-)$ are functors in $\AC$ (resp. $\MC$).}
\end{enumerate}
\end{theorem}
\begin{proof} (\ref{Obj}.) Since every cone in $\AC$ can be endowed (non-canonically) with a grading and a section it is sufficient to show that the above operations applied to cones in $\MC$ lead to cones in $\MC$. Thus assume $(C_i,g_i,s_i)\in \MC$ for $i=1,2$. Since $C_2$ is closed, $\Hom_\AC(C_1,C_2)=\bigcap_{c\in C_1} \bigcap_{\lambda\in C_2^*} \{f: \lambda(f(c))\geq 0\}$ and thus it is an intersection of closed sets and hence closed. Since $s_2(1)$ is in the interior of $C_2$, the cone $\Hom_\AC(C_1,C_2)$ contains the element  $s_2\circ g_1: C_1\rightarrow C_2$ as well as any homomorphism $f+s_2\circ g_1$ for $f$ in a sufficiently small ball around the origin in $\Hom(LC_1,LC_2)$ and thus  $\Hom_\AC(C_1,C_2)$ is full-dimensional. Finally the function $h$ sending $f\in \Hom(LC_1,LC_2)$ to $g_2(f(s_1(1)))$ defines a grading since $h(f)=0$ implies that $f(s_1(1))=0$ and thus $f$ maps an interior point of $C_1$ to $0$ forcing $f$ to be the $0$ map. Thus $\left(\Hom_\AC(C_1,C_2), h, s_2\circ g_1\right)\in \MC$. For the tensor product note that 
$C_1\cap g_1^{-1}(1)\times C_2\cap g_2^{-1}(1)\rightarrow C_1\cap g_1^{-1}(1)\otimes C_2\cap g_2^{-1}(1)$ is a continuous surjection and thus the right hand side is compact and does not contain the origin. It follows that the cone over it, which is $C_1\otimes C_2$ is closed. It is full-dimensional since the pairwise tensor products of bases for $LC_1$ and $LC_2$ contained in $C_1$ and $C_2$ resp. are a basis for $LC_1\otimes LC_2$ contained in $C_1\otimes C_2$. Similarly $s_1(1)\otimes s_2(1)$ is an interior point of $C_1\otimes C_2$. Now let $\beta\in C_1\otimes C_2$ so $\beta = \sum c_1^j\otimes c_2^j$ with $c_i^j\in C_1$ and note that $(g_1\otimes g_2)(\beta)=\sum g_1(c_1^j)\otimes g_2(c_2^j)=0$  iff for every summand either $c_1^j$ or $c_2^j$ are zero and thus iff $\beta=0$. As a result the function $g_1\otimes g_2$ is a grading of the tensor product. It follows that $\left( C_1\otimes C_2, g_1\otimes g_2, s_1\otimes s_2\right)\in \MC$. Here we implicitly used the fact that the multiplication map gives a canonical isomorphism between $\rr\otimes \rr$ and $\rr$. For the symmetric powers define the function $\Sym^p(g_1)(v_1\dots v_p)=g_1(v_1)\dots g_1(v_p)\in \Sym^p(\rr)\cong \rr$ and extend linearly. It is immediate that $\Sym^p(g_1)$ is a grading on $\Sym^p(C_1)$. The multiplication map $\mu: C_1^{\otimes p}\rightarrow \Sym^p(C_1)$ mapping $v_1\otimes\dots \otimes v_p\rightarrow v_1\dots v_p$ is a surjective linear map whose image is the cone over  the compact set $\mu\left((C_1\cap g_1^{-1}(1))^{\otimes p}\right)$ which is contained in $\Sym^p(g_1)^{-1}(1)$ and thus does not contain the origin. It follows that $\Sym^p(C_1)$ is a closed and full-dimensional cone. It follows that $\left(\Sym^p(C_1), \Sym^p(g_1), \Sym^p(s_1)\right)\in \MC$. Here we have implicitly used the fact that $\Sym^p(\rr)$ is canonically isomorphic to $\rr$ via the multiplication. (\ref{functors}.) Define the operations on objects as in part $(\ref{Obj}.).$ Since $\Hom_\AC(C_1,C_2)\subseteq \Hom(LC_1,LC_2)$ then we can define the action of $\Hom_{AC}(C_1,-)$, $\Hom_{\AC}(-, C_2)$, $C_1\otimes -$ and $\Sym^p(-)$ on morphisms as that of the corresponding functors on vector spaces and this definition will respect compositions. It follows that the above operations are functors in $\AC$. For functoriality in $\MC$ we need to verify that the images of morphisms in $\MC$ are also in $\MC$ (i.e. commute with the grading and the section of the corresponding objects). We verify the case of $\Hom_\AC(D,-)$ and leave the remaining similar verifications to the reader. Thus assume $f\in \Hom_\MC(A,C)$ and we wish to verify that $\hat{f}:\Hom_\AC(D,A)\rightarrow \Hom_\AC(D,C)$ is a morphism of marked cones. This amounts to showing that the equalities $g_C\circ f\circ h \circ s_D(1)= g_A\circ h\circ s_D(1)$ and $f\circ s_A\circ g_D=s_C\circ g_D$ hold for every $h\in \Hom_\AC(D,A)$. This is an immediate consequence of the fact that $f\in \Hom_\MC(A,C)$. 
\end{proof}

\begin{remark} Via the equivalence in Lemma~\ref{Equiv} the above functors define operations on convex bodies containing the origin in their interior. Concretely, for convex bodies $P_i\subseteq E_i$ we have $P_1\otimes P_2:={\rm Conv}\{(p_1,p_2,p_1\otimes p_2):p_i\in P_i\}\subseteq E_1\times E_2\times (E_1\otimes E_2)$. For an integer $n>0$ and $0\leq j\leq n$ let $e_j(x_1,\dots, x_n)$ be the $j$-th elementary symmetric polynomial in $n$ variables. We have  
$\Sym^n(P_1):={\rm Conv}\{( e_1(p_1,\dots,p_n),\dots, e_n(p_1,\dots,p_n)): p_i\in P_1\}\subseteq \prod_{j=1}^n \Sym^j(E_1)$.
\end{remark}

Next we show that tensor powers and symmetric products are solutions to universal linearization problems.

\begin{definition} Let $n>0$ be an integer and $C_1,\dots, C_n,D\in \AC$. A function $T:\prod C_i\rightarrow D$ is multilinear if it is the restriction of a multilinear function $T:\prod LC_i\rightarrow LD$ satisfying $T(\prod C_i)\subseteq D$. Note that $T$ is uniquely determined by its restriction to $\prod C_i$. A multilinear $T$ is symmetric if for every permutation $\sigma\in S_n$ and every $v_i\in C_i$ we have $T(v_{\sigma(1)},\dots, v_{\sigma(n)})=T(v_1,\dots, v_n)$. 
\end{definition}

\begin{theorem} \label{Lemma: Univ} Let $n>0$ be an integer. 
\begin{enumerate}
\item{The following universal linearization properties hold,
\begin{enumerate}
\item{\label{U1} Assume $C_1,\dots, C_n\in \AC$ and let $u:\prod C_i\rightarrow \otimes C_i$ be the map $u(v_1,\dots, v_n)=v_1\otimes\dots\otimes v_n$. For every $D\in \AC$ and every multilinear map $T:\prod C_i\rightarrow D$ there is a unique $t\in \Hom_\AC(\bigotimes C_i,D)$ such that $T=t\circ u$.}
\item{\label{U2} Let $q: C_1^{n}\rightarrow \Sym^n(C_1)$ be the map $q(v_1,\dots, v_p)=v_1\dots v_n$. For every $D\in \AC$ and every multilinear symmetric map $T:C_1^n\rightarrow D$ there is a unique $t\in \Hom_\AC(\Sym^n(C_1),D)$ such that $T=t\circ q$.}
\end{enumerate}
}
\item{\label{Adjoint} For every $A,B,C\in \AC$, tensor products and homs satisfy the following adjunction formula $\Hom_\AC(A\otimes B, C)\cong \Hom_\AC(A,\Hom_\AC(B,C)).$
} 
\end{enumerate}
\end{theorem}
\begin{proof} (\ref{U1}.) By the universal property of tensor products in the category of vector spaces every such $T$ determines a unique linear map $t:\bigotimes LC_i\rightarrow LD$ with $T=t\circ h$ thus every generator $v_1\otimes \dots \otimes v_n$ of the cone $\bigotimes C_i$, and hence the cone itself is mapped via $t$ to $D$. Conversely a morphism $t:\bigotimes C_i\rightarrow D$ is the restriction of a unique linear map $t:\bigotimes LC_i\rightarrow LW$ and thus defines a multilinear map $T:\prod LC_i\rightarrow LD$ via $T=t\circ h$. Since $t(\bigotimes C_i)\subseteq D$ it follows that $T(\prod C_i)\subseteq D$ as claimed. (\ref{U2}.) Follows by a similar argument from the universal property of symmetric powers in the category of vector spaces. (\ref{Adjoint}.) Follows from part (~\ref{U1}.) since $\Hom_\AC(A,\Hom_\AC(B,C))$ is the set of bilinear maps from $A\times B$ to $C$. \end{proof}

The following Theorem describes some basic properties of the facial structure of the cones obtained by applying linearization functors,

\begin{theorem} \label{Faces} Let $A,B\in \AC$ and $n>0$ an integer. The following statements hold:
\begin{enumerate}
\item{\label{tensor} The extreme rays of $A\otimes B$ are precisely the tensor products of extreme rays of $A$ and $B$. If $F_A$ and $F_B$ are faces (resp.exposed faces) of $A$ and $B$ then $F_A\otimes F_B$ is a face (resp. an exposed face) of $A\otimes B$.}
\item{\label{sym} The extreme rays of $\Sym^n(A)$ are products of extreme rays of $A$. If $F_1,\dots, F_n$ are exposed face of $A$ then $F_1\cdots F_n:={\rm Cone}(p_1\cdots p_n:p_i\in F_i)$ is an exposed face of $\Sym^n(A)$.}
\item{\label{hom} The maximal exposed faces of $\Hom_\AC(A,B)$ are in canonical correspondence with the exposed extreme rays of $A\otimes B^*$.}
\end{enumerate}
\end{theorem}
\begin{proof} (\ref{tensor}.) By the Krein-Milman Theorem the cones $A$ and $B$ are generated by their extreme rays. As a result $A\otimes B$ is generated by the tensor powers of extreme rays and thus every extreme ray of $A\otimes B$ is of this form. Let $P_A:=A\cap g_A^{-1}(1)$ and $P_B:=B\cap g_A^{-1}(1)$ and recall that the convex set $P_A\otimes P_B={\rm Conv}\{(p_A,p_B,p_A\otimes p_B,1):p_A\in P_A,p_B\in P_B\}$ is precisely $A\otimes B\cap (g_a\otimes g_B)^{-1}(1)$ and in particular there is a correspondence between the faces of $P_A\otimes P_B$ and the nonempty faces of $A\otimes B$. Now suppose $F_A$ and $F_B$ are faces of $A$ and $B$ inducing faces $G_A$ and $G_B$ of $P_A$ and $P_B$. We want to show that the convex set $G_A\otimes G_B:={\rm Conv}\{(p_A,p_B,p_A\otimes p_B,1):p_A\in G_A,p_B\in G_B\}$ is a face of $P_A\otimes P_B$. If $a_i\in P_A$ and $b_i\in P_B$ and $\sum \lambda_i(a_i,b_i,a_i\otimes b_i,1)\in G_A\otimes G_B$ then $\sum \lambda_i a_i\in G_A$ and $\sum \lambda_ib_i\in G_B$ and thus, since $G_A$ and $G_B$ are faces of $P_A$ and $P_B$, we have that $a_i\in G_A$ and $b_i\in G_B$. It follows that $(a_i,b_i,a_i\otimes b_i,1)\in G_A\otimes G_B$. In particular the tensor products of extreme rays of $A$ and $B$ are extreme rays of $A\otimes B$. If $\phi_A$ and $\phi_B$ are supporting linear functions for $F_A$ and $F_B$ then $\phi_A\otimes g_B + g_A\otimes \phi_B$  is a linear functional supporting the face $F_A\otimes F_B$. (\ref{sym}.) By the Krein-Milman Theorem the cone $A$ is generated by its extreme rays and thus $\Sym^n(A)$ is generated by the products of extreme rays of $A$ and in particular every extreme ray of $\Sym^n(A)$ must be a product of extreme rays of $A$. If $\phi_i$ is a linear functional supporting the face $F_i$ then the linear function associated to the symmetric multilinear map $\psi:LV^n\rightarrow \rr$  given by $\psi(v_1,\dots, v_n):=\sum_{i=1}^n g_A(v_1)\dots g_A(v_{i-1})\phi_i(v_i)g_A(v_{i+1})\dots g_A(v_n)$ supports $F_1\cdots F_n$. (\ref{hom}.) By Lemma~\ref{Lemma: Univ} part (\ref{Adjoint}.) we have $\Hom_\AC(A,B)=\Hom_\AC(A,\Hom_\AC(B^*,\rr))=\Hom_\AC(A\otimes B^*,\rr)$. The result now follows from the well known correspondence between the maximal proper exposed faces of a convex cone and the exposed rays of its dual.
\end{proof}

\begin{example} Let $P=[-1,1]\subseteq \mathbb{R}$. The set $P\otimes P$ is the simplex shown in Figure~\ref{Fig1}. The only proper nonempty faces of $P\otimes P$ which are tensor powers of faces of $P$ are the vertices and the edges drawn in bold in the figure. If $n>0$ is an integer then $\Sym^n(P)$ is the simplex in $\rr^n$ whose vertices are obtained as the coefficients of the positive powers of $x$ in the polynomial $(x+1)^m(1-x)^{n-m}$ for $0\leq m\leq n$. In particular $\Sym^2(P)$ is the triangle shown in Figure~\ref{Fig1}. The only proper nonempty faces of $\Sym^2(P)$ which can be obtained as products of faces of $P$ are the vertices and bold edges. More generally there are $\binom{n+2}{2}$ product faces among the $2^n-1$ nonempty faces of $\Sym^n(P).$
\begin{figure}[h]
\tdplotsetmaincoords{30}{10}
\begin{tikzpicture}[tdplot_main_coords,scale=1.3]
\coordinate (V1) at (1,1,1);
\coordinate (V2) at (-1,1,-1);
\coordinate (V3) at (-1,-1,1);
\coordinate (V4) at (1,-1,-1);
\draw[fill=cof,opacity=0.6] (V1) -- (V3) -- (V4);
\draw[fill=pur,opacity=0.6] (V3) -- (V1) -- (V2);
\draw[fill=greeo,opacity=0.6] (V4) -- (V1) -- (V2);
\draw[fill=greet,opacity=0.6] (V2) -- (V3) -- (V4);
\draw [line width=1.7] (V1)--(V2)--(V3)--(V4)--cycle;
\draw[thick,dashed,->] (-2,0,0) -- (2,0,0) node[anchor=north east]{$x$};
\draw[thick,dashed, ->] (0,-2,0) -- (0,2,0) node[anchor=north west]{$y$};
\draw[thick,dashed, ->] (0,0,-2) -- (0,0,2.5) node[anchor=south]{$z$};
\node [above] at (1,1,1) {$(1,1,1)$};
\node [below] at (1,-1,-1) {$(1,-1,-1)$};
\node [left] at (-1,-1,1) {$(-1,-1,1)$};
\node [left] at (-1,1,-1) {$(-1,1,-1)$};
\end{tikzpicture}
\begin{tikzpicture}[scale=1.3]
\coordinate (V1) at (-1,0);
\coordinate (V2) at (1,2);
\coordinate (V3) at (1,-2);
\draw[line width=1.7,fill=cof] (V1) -- (V2);
\draw[line width=1.7,fill=cof] (V1) -- (V3);
\draw[fill=greeo,opacity=0.6] (V3) -- (V1) -- (V2)--cycle;
\node at (-1.5,0.2) {$(-1,0)$};
\node at (1,2.2) {$(1,2)$};
\node at (1,-1.8) {$(1,-2)$};
\draw[thick,dashed,->] (-2,0) -- (2,0) node[anchor=north east]{$x$};
\draw[thick,dashed, ->] (0,-2) -- (0,2) node[anchor=north west]{$y$};
\end{tikzpicture}
\caption{ $P\otimes P$ and $\Sym^2(P)$ for $P=[-1,1]$.}
\label{Fig1}
\end{figure}
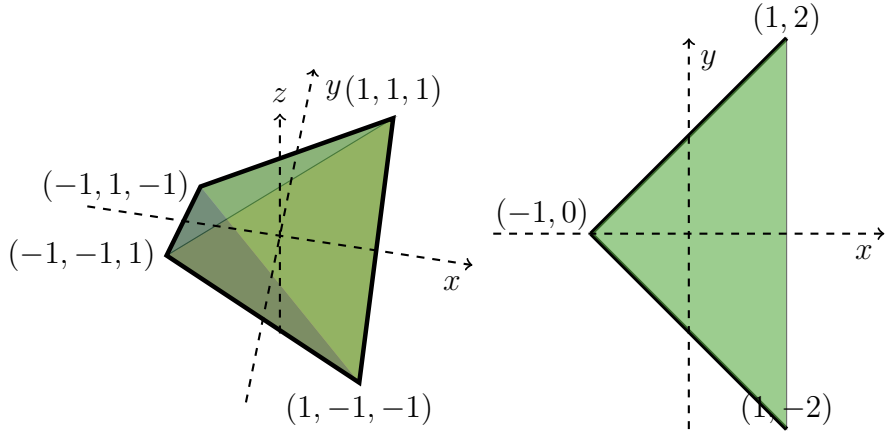
As an illustration of Lemma~\ref{Lemma: Univ} note that for $a_i\in \rr$ we have 
\[\max_{(x_1,x_2)\in P\times P} a_1x_1x_2+a_2x_1+a_3x_2+a_4 = \max_{(x,y,z)\in P\otimes P} a_1x+a_2y+a_3z+a_4.\] 
\end{example}
It is thus possible to linearize certain nonlinear optimization problems via tensor products. In exchange the domain of the problem has been modified and thus the usefulness of this approach is limited by whether or not we have a description of the tensor product amenable to efficient computation. We address these questions in the following section.

\begin{example} \label{TSP} Consider the following natural extensions of the symmetric traveling salesman problem,
\begin{enumerate}
\item{ {\bf Two-tier TSP:} A company has two kinds of traveling salesmen who wish to visit $n$ cities. The travelers have different costs of traveling $c_{ij}=c_{ji}$ and $d_{ij}=d_{ji}$ between cities $i$ and $j$. Moreover, the airline gives the company a discount $-u_{ij}$ whenever both salesmen choose to take a trip between cities $i$ and $j$. The company wishes to find itineraries for both travelers so as to minimize the total cost.}
\item{ {\bf Roman army TSP:} The roman army has $m$ identical legions patrolling $n$ cities with identical traveling cost $c_{ij}=c_{ji}$ between cities $i$ and $j$. The emperor wants the legions to go over as many of the roman roads as possible. Thus two or more legions traveling along the same road have a penalty $p^{r}_{ij}$ whenever $r>0$ legions choose to use the road joining $i,j$. Find routes for all legions so as to minimize total cost.}
\end{enumerate} 
Let $STSP(n)$ be the symmetric traveling salesman polytope (i.e. the convex hull of the adjacency matrices of all Hamiltonian cycles in the complete graph). Choose coordinates for the span of $STSP(n)$ with center on the average of its vertices and obtain a convex body which we also denote as $STSP(n)$. By Lemma~\ref{Lemma: Univ} the two problems above can be linearized on the polytopes $2TSP(n):=STSP(n)\otimes STSP(n)$ and $RTSP(m,n):=\Sym^m(STSP(n)).$ 
\end{example}
The polytopes in the previous example satisfy a remarkable universality property,
\begin{theorem} Fix $m>1$. If $P\subseteq \rr^d$ is a $\{0,1\}$-polytope (i.e. all the components of its vertices are in $\{0,1\}$) then there is an integer $N(d)>0$ such that $P$ is isomorphic to a face of  $2STSP(n)$ and of $RTSP(m,n)$.
\end{theorem}
\begin{proof} By a result of Billera and Sarangarajan~\cite[Theorem 3.1]{BiS} $P$ is isomorphic to a face of the asymmetric traveling salesman polytope $ASTP(N)$  for some $N$. By a result of Karp~\cite{K} the $ASTP(N)$ appears as a face of $STSP(2N)$ and by Theorem~\ref{Faces} $STSP(2N)$ appears as a face of $2STSP(2N)$ (resp. $RTSP(m,2N)$) by taking the tensor product (resp. the product) of $STSP(2N)$ and any vertex (resp. and the $(m-1)$-st power of any vertex).\end{proof}

\section{Approximating convex hulls of polynomial images in the presence of a measure.}\label{Sec: Approx}

With the purpose of computing linearization functors we introduce an approximation scheme for convex hulls of polynomial images of compact sets via projections of spectrahedra of interest in its own right.
Let $B\subseteq \rr^n$ be a compact set, let $m$ be a finite measure supported on $B$ (i.e. every open neighborhood of every point of $B$ has positive $m$ measure) and let $T:\rr^n\rightarrow \rr^m$ be a polynomial map with components $T_i(x_1,\dots,x_n)$, $1\leq i\leq m$. The following approximation method for the convex hull of $T(B)$ is an extension of the ideas of Barvinok and Veomett~\cite{BV} on semidefinite approximations of convex sets.
 
\begin{definition} For $\lambda\in \rr^m$ and $k\geq 0$ define the symmetric bilinear form $\phi^k_\lambda$ on the vector space of polynomials of degree at most $k$ in $\rr^n$ by 
\[\phi^k_\lambda(a(x),b(x)):=\large\int_B \left(1+\sum_{i=1}^m \lambda_iT_i(x)\right)a(x)b(x)dm(x).\]
\end{definition}

\begin{lemma} \label{Approx} Define $Q_k:=\{\lambda\in \rr^m: \phi_\lambda^k\succeq 0\}$. The following statements hold:
\begin{enumerate}
\item{ \label{spectrahedron} For every $k$, $Q_{k+1}\subseteq Q_k$, the set $Q_k$ is a spectrahedron and $Q_k^{\circ}$ is an SDR set.}
\item{ \label{inclusion} ${\rm Conv}(T(B))^{\circ}\subseteq Q_k$ for every $k$.}
\item{ \label{equality} ${\rm Conv}(T(B))^{\circ}=\bigcap_{k=1}^\infty Q_k$.}
\item{ \label{density} ${\rm Conv}(T(B))=\overline{\bigcup_k (Q_k^{\circ})}.$}
\item{ \label{Finite} If $B$ is a finite set and $m$ the counting measure on $B$ then ${\rm Conv}(T(B))=Q_k^{\circ}$ for some $k$.}
\end{enumerate}
\end{lemma}
\begin{proof}  (\ref{spectrahedron}.) The inclusion occurs because a polynomial of degree at most $k$ also has degree at most $k+1$. The set $Q_k$ is a spectrahedron because the function $\phi_\lambda^k$ is affine linear in $\lambda$. By~\cite[Proposition 3.2]{GN}, $Q_k^{\circ}$ is the projection of a spectrahedron. (\ref{inclusion}.) An affine linear function is nonnegative on a set $S$ iff it is nonnegative in its convex hull. As a result   
${\rm Conv}(T(B))^{\circ}=\{\lambda\in \rr^m: \forall x\in B\left(\sum_i \lambda_i T_i(x)\geq -1\right)\}$. If $\lambda\in {\rm Conv}(T(B))^{\circ}$, $k\in \mathbb{N}$ and $a(x)$ is any polynomial of degree at most $k$ then the quantity $\int_B \left(1+\sum_{i=1}^m \lambda_iT_i(x)\right)a^2(x)dm(x)$ is the integral of a nonnegative function on $B$ and thus nonnegative. Hence $\lambda\in Q_k$. (\ref{equality}.) Suppose $\lambda\in \bigcap_{k=1}^{\infty} Q_k$ and let $\delta(x):=1+\sum \lambda_iT_i(x)$. If $\lambda\not\in {\rm Conv}(T(B))^{\circ}$ then there exists $x^*\in B$ such that $\delta(x^*)<0$. By continuity of $\delta$ there exists radii $0<r_1<r_2$  and balls $B_{r_i}(x^*)$ such that $\delta(x)<0$ in $B_{r_2}(x^*)\cap B$. Let $h(x)$ be a continuous function such that $h(x)=1$ in $B_{r_1}(x)$ and $h(x)=0$ outside $B_{r_2}(x^*)$. By the Stone-Weierstrass theorem there is a sequence of polynomials $s_n(x)$ converging uniformly to $h(x)$ in $B$ and by finiteness of $m$ and the dominated convergence theorem 
\[ \lim_{n\rightarrow \infty}\int_B \delta(x) s_n(x)^2dm=\int_B \delta(x)h(x)^2dm<0\]
it follows that for some sufficiently large $N$ the integral $\int_B \delta(x) s_N(x)^2dm<0$ contraddicting the fact that $\phi_\lambda^{\deg (s_N^2) }$ is positive semidefinite.
(\ref{density}.) By part (\ref{equality}.) we have the equalities
\[\left(\bigcup_{k=1}^\infty Q_k^{\circ}\right)^{\circ} = \bigcap_{k=1}^{\infty}(Q_k^{\circ})^{\circ} = \bigcap_{k=1}^{\infty} Q_k = {\rm Conv}(T(B))^{\circ}\]
Taking polars and using the fact that ${\rm Conv}(T(B))$ is closed we obtain the claimed equality. (\ref{Finite}.) Let $B=\{b_1,\dots, b_s\}$. Since $B$ is finite the affine coordinate ring $\rr[B]$ is a product of fields and thus there exist interpolation  polynomials $p_i(x)$ such that $p_i(b_j)=\delta_{ij}$ of degree at most $k$. If $\lambda \in Q_k$ then $1+\sum \lambda_iT_i(b_j)=\phi_{\lambda}^k(p_j,p_j)\geq 0$ for $1\leq j\leq s$ so $\lambda\in {\rm Conv}(T(B))^{\circ}$ and the claim follows by polarity.  
\end{proof}
\begin{remark} By Lemma~\ref{Approx} part $(\ref{density})$ we call the above construction an accurate approximation scheme for ${\rm Conv}(T(B))$ via SDR sets. Note that if $B$ is a convex body then the restriction of the Lebesgue measure to $B$ satisfies the necessary hypothesis for the above approximation scheme. 
\end{remark} 
\begin{remark} Via positivstellensatz-type results it is possible to obtain approximations of ${\rm Conv}(T(B))$ from above for some compact sets $B$ (for instance basic closed semialgebraic). However, this approach has two drawbacks. First, it does not apply to arbitrary convex bodies $B$ and second, when the corresponding quadratic module is not finitely generated it is not clear how to write the resulting approximating convex bodies as projections of spectrahedra. Because of these two reasons we find the above approximation scheme preferable. A case of much interest when the corresponding quadratic module is not finitely generated is for instance the positivestellensatz for projections of spectrahedra of Gouveia and Netzer~\cite[Theorem 5.1]{GN}.
\end{remark}

As an application of Lemma~\ref{Approx} we obtain accurate approximation schemes for linearization functors.
\begin{theorem} Let $A,B\in \MC$ and $n>0$ an integer and let $P_A:=A\cap g_A^{-1}(1)$ and $P_B:=B\cap g_B^{-1}(1)$.
\begin{enumerate}
\item{Assume $m_A,m_B,m_{B^{\circ}}$ are finite measures supported in $P_A, P_B$ and $P_{B^{\circ}}$. These measures determine accurate approximation schemes for $A\otimes B$, $\Sym^n(A)$ and $\Hom_\AC(A,B)$ by SDR sets.}
\item{If $A$ and $B$ are polyhedra and $m_A$, $m_B$ and $m_{B^\circ}$ are the counting measures on the vertices of $P_A, P_B$ and $P_{B^\circ}$ respectively then the approximation schemes converge after finitely many steps and we obtain projected spectrahedral representations of the polyhedra  $A\otimes B$, $\Sym^n(A)$ and $\Hom_\AC(A,B)$.}
\end{enumerate}
\end{theorem} 
\begin{proof} Let $E_A, E_B$ be the vector space structures on $g_A^{-1}(1)$ and $g_B^{-1}(1)$ obtained by placing the origin in $s_A(1)$ and $s_B(1)$ respectively. Let $T: E_A\times E_B\rightarrow E_A\times E_B\times(E_A\otimes E_B)$ be the map sending $(a,b)\rightarrow (a,b,a\otimes b)$ and note that this is a polynomial map satisfying $P_A\otimes P_B={\rm Conv}(T(P_A\times P_B))$. Since $m_A\times m_B$ is a finite measure on the compact set $P_A\times P_B$ the Lemma~\ref{Approx} gives us an accurate approximation scheme for the tensor product via SDR sets. The corresponding $\MC$ cones approximate the tensor product $A\otimes B$ as claimed. Similarly $\Sym^n(P_A)\subseteq \bigoplus_{j=1}^n \Sym^j(E_A)$ is the convex hull of the image of $P_A^n$ under the polynomial map $T(p_1,\dots,p_n)=(e_n(p_1,\dots, p_n),\dots, e_1(p_1,\dots, p_n))$ where the $e_i$ the elementary symmetric polynomials in $n$ variables. The product measure $m_A^{n}$ supported on $P_A^n$ allows us, via Lemma~\ref{Approx} to construct an accurate approximation scheme for symmetric powers via SDR sets. The corresponding $\MC$ cones approximate $\Sym^n(A)$ as claimed.
Now, by Theorem~\ref{Lemma: Univ} part (\ref{Adjoint}.) we know  $\Hom_\AC(A,B) = (A\otimes B^*)^*$ and the statement follows form the accurate approximation scheme for tensor products from the first paragraph. In particular we can write $\Hom_\AC(A,B)$ as a countable intersection of spectrahedra. By Theorem~\ref{Faces} the vertices of the polytopes $P_A\otimes P_B$ and $\Sym^n(P_A)$ are precisely products of vertices. It follows that these convex sets are the convex hulls under polynomial maps of finite sets and the conclusion follows by Lemma~\ref{Approx} part (\ref{Finite}.). Since $\Hom_\AC(A,B) = (A\otimes B^*)^*$ the accurate approximation scheme for $A\otimes B^*$ allows us to give a spectrahedral description of $\Hom_\AC(A,B)$.
\end{proof}

\section{Computing linearization functors}\label{Sec: Computing}

Next we study how our operations behave under morphisms. As a result we obtain exact computations for tensor powers, symmetric powers and Homs of SDR cones.
\begin{definition}  Let $f\in \Hom_\AC(C,D)$. The morphism is strongly injective if the linear function $f:LC\rightarrow LD$ is injective and satisfies $f^{-1}(D)=C$. We denote strongly injective morphisms as $f:C \hookrightarrow D$.The morphism has dense image if $\overline{f(C)}=D$. We denote it by $f:C\twoheadrightarrow D$.
\end{definition}
Recall that for cones in $\AC$ there are morphisms with non-closed image while in $\MC$ a morphism with dense image is surjective (see Section~\ref{Categories}). Strongly injective morphisms and morphisms with dense image are closely related via duality,
\begin{lemma}\label{SIDI} $f:C\hookrightarrow D$ if and only if $f^*:D^*\twoheadrightarrow C^*$.\end{lemma}
\begin{proof} Assume $f:C\hookrightarrow D$ and let $h$ be a linear functional nonnegative in $f^*(D^*)$. We will show that $h$ must be nonnegative in $C^*$ and conclude that $\overline{f^*(D^*)}=C^*$.  Via the canonical identification between $((LC)^*)^*$ and $LC$ we can assume that $h: (LC)^*\rightarrow \rr$ is evaluation at a point $p\in LC$. By our assumption for every $g\in D^*$ $g(f(p))\geq 0$ and thus $f(p)\in D$. Since $f$ is strongly injective this implies that $p\in D$ and thus $h$ is nonnegative in $C^*$ as claimed.
For the converse we will show that if $f: C\twoheadrightarrow D$ then $f^*:D^*\hookrightarrow C^*$ and conclude that the claim holds by bi-duality. Suppose that $h\in (LD)^*$ and that $f^*: (LD)^*\rightarrow (LC)^*$ maps $h$ to $C^*$. It follows that for every $c\in C$ $h(f(c))\geq 0$ and since $f(C)$ is dense in $D$ that $h\in D^*$ so $f^*$ is strongly injective as claimed.
\end{proof}
If $f\in \Hom_\AC(C,D)$ and $A_1\in \AC$ then we denote by $\hat{f}: \Hom_\AC(A_1,C)\rightarrow \Hom_\AC(A_1,D)$ the morphism obtained by composition with $f$. Using this notation we have,
\begin{theorem} \label{morp} For $i=1,2$ let $A_i,B_i$ and $C$ be cones in $\AC$ and let $n>0$ be an integer. The following statements hold:
\begin{enumerate}
\item{ \label{DI} If $f:A_1\twoheadrightarrow B_1$ then $\Sym^n(f): \Sym^n(A_1)\twoheadrightarrow \Sym^n(B_1)$, $f^*:\Hom_\AC(B_1,C)\hookrightarrow \Hom_\AC(A_1,C)$ and $\hat{f}^*:\Hom_\AC(C,B_1^*)\hookrightarrow \Hom_\AC(C,A_1^*)$. If moreover $g:A_2\twoheadrightarrow B_2$ then $f\otimes g:A_1\otimes A_2\twoheadrightarrow B_1\otimes B_2$.}
\item{ \label{SI} If $h:A_1\hookrightarrow B_1$ then $\hat{h}:\Hom_\AC(C,A_1)\hookrightarrow \Hom_\AC(C,B_1)$, $h^{**}:\Hom_\AC(A_1^*,C)\hookrightarrow \Hom_\AC(B_1^*,C)$, $\Sym^n(h):\Sym^n(A_1^*)^*\hookrightarrow \Sym^n(B_1^*)^*$ and we have inclusions $\Sym^n(A_1)\subseteq \Sym^n(h)^{-1}(\Sym^n(B_1))\subseteq \Sym^n(A_1^*)^*$. If moreover $t:A_2\hookrightarrow B_2$ then $h\otimes t: (A_1^*\otimes A_2^*)^*\hookrightarrow (B_1^*\otimes B_2^*)^*$ and we have inclusions 
$A_1\otimes A_2\subseteq (h\otimes t)^{-1}(B_1\otimes B_2)\subseteq (A_1^*\otimes A_2^*)^*$.}
\end{enumerate}
\end{theorem}
\begin{proof} (\ref{DI}.) If $A_i\twoheadrightarrow B_i$ then $\Sym^n(A_1)\twoheadrightarrow \Sym^n(B_1)$ and $A_1\otimes A_2\twoheadrightarrow B_1\otimes B_2$ because by our assumption the images of the generators of the domain are dense in the generators of the codomain. If $f: A_1\twoheadrightarrow B_1$ and $g\in \Hom(LB_1,LC)$ satisfies $f^*(g) \in \Hom_\AC(A_1,C)$ then for every $a\in A_1$ we have $g(f(a))\in C$. Since $f(A_1)$ is dense in $B_1$ it follows that $g(y)\in C$ for $y\in B_1$ so $g\in \Hom_\AC(B_1,C)$ and thus $f^*$ is strongly injective as claimed. Finally, by Lemma~\ref{SIDI} we have $f^*:B_1^*\hookrightarrow C_1^*$ and using the first claim of part (\ref{SI}.) to be shown next we conclude $\hat{f}^*:\Hom_\AC(C,B_1^*)\hookrightarrow \Hom_\AC(C,A_1^*)$.  (\ref{SI}.) Assume $h: A_1\hookrightarrow B_1$ and suppose $g\in \Hom(LC,LA_1)$ is such that $\hat{h}(g)\in \Hom_\AC(C,B_1)$. It follows that for every $c\in C$ $h(g(c))\in B_1$ so $g(c)\in B_1$ since $h$ is strongly injective so $g\in \Hom_\AC(C,A_1)$ as claimed. By Lemma~\ref{SIDI} we have $h^*:B_1^*\twoheadrightarrow A_1^*$ and by part (\ref{DI}) that $h^{**}: \Hom_\AC(A_1^*,C)\hookrightarrow \Hom_\AC(B_1^*,C)$.  Again by part (\ref{DI}.) our hypothesis on $h$ yields $\Sym^n(h^*): \Sym^n(B_1^*)\twoheadrightarrow \Sym^n(A_1^*)$ and by Lemma~\ref{SIDI} $\Sym^n(h):\Sym^n(A_1^*)^*\hookrightarrow \Sym^n(B_1^*)^*$. 
By Lemma~\ref{SIDI} and part (\ref{DI}.) we have $h^*\otimes t^*: B_1^*\otimes B_2^*\twoheadrightarrow A_1^*\otimes A_2^*$ and applying Lemma~\ref{SIDI} again we conclude $h\otimes t: (A_1^*\otimes A_2^*)^*\hookrightarrow (B_1^*\otimes B_2^*)^*$. The claimed inclusions are immediate.
\end{proof}

The most interesting feature of the previous Theorem is the behavior of tensor products and symmetric powers under strongly injective maps in part (\ref{SI}). The inclusions in the Theorem suggest a sort of failure of ``left exactness'' of tensors and symmetric powers which makes them difficult to compute. The following example shows that the inclusions may be strict even for polyhedra.

\begin{example} Denote $\rr^n$ with canonical basis $a_1,\dots, a_n$ with the symbol $\rr^n_a$ and let $a_1',\dots, a_n'$ be its dual basis. Let $C_x:={\rm Cone}(\pm x_1\pm x_2+x_3) \subseteq \rr^3_x$ and let $\phi_x:C_x\hookrightarrow \rr^4_+$. We know that $C_x\otimes C_y \subseteq (\phi_x\otimes\phi_y)^{-1}(\rr^4_+\otimes \rr^4_+)=(C_x^*\otimes C_y^*)^*$ and we will show that the inclusion is strict.

By Theorem~\ref{morp} part (\ref{SI}.) we know  $\phi_x\otimes \phi_y: (C_x^*\otimes C_y^*)^*\hookrightarrow ((\rr_+^4)^*\otimes (\rr_+^4)^*)^*\subseteq \rr^4_z\otimes \rr^4_w$. Ordering the $z_i\otimes w_j$ of the tensor product lexicographically our morphisms can be represented by the matrices
\[
\phi =
\left(
\begin{array}{ccc}
-1 & 0 & 1\\
1 & 0 & 1\\
0 & -1 & 1\\
0 & 1 & 1\\
\end{array}
\right),
\phi\otimes \phi = \left(
\begin{array}{ccc}
-\phi & 0 & \phi\\
\phi & 0 &\phi\\
0 & -\phi & \phi\\
0 & \phi & \phi\\
\end{array}
\right)
\]
Let $\beta:=-x_1\otimes y_1+x_1\otimes y_2+x_2\otimes y_1+x_2\otimes y_2+x_3 \otimes y_3$ and note that $\beta \in (C_x^*\otimes C_y^*)^*$ because $(\phi\otimes \phi)(\beta)=2(z_2\otimes w_1+z_3 \otimes w_1+z_1\otimes w_2+ z_4\otimes w_2+z_1\otimes w_3+z_3\otimes w_3+z_2\otimes w_4+ z_4\otimes w_4)$ has nonnegative coefficients. We will show that $\beta\not\in C_x\otimes C_y$. To this end let $\gamma:=x_1'\otimes y_1'-x_1'\otimes y_2'-x_2'\otimes y_1'-x_2'\otimes y_2'+2x_3'\otimes y_3'$ and note that $\gamma((\epsilon_1x_1+\epsilon_2x_2+x_3)\otimes (\delta_1y_1+\delta_2y_2+y_3))=\epsilon_1\delta_1-\epsilon_1\delta_2-\epsilon_2\delta_1-\epsilon_2\delta_2+2$ which cannot be negative for $\epsilon_i,\delta_i\in\{-1,1\}$. It follows that $\gamma\in (C_x\otimes C_y)^*$. Since $\gamma(\beta)=-2<0$ we conclude $\beta\not \in C_x \otimes C_y$ as claimed.
\end{example}
\begin{example} Continuing with the previous example, we know that $\Sym^2(C_x)\subseteq \Sym^2(\phi)^{-1}(\Sym^2(\rr_+^4))=\Sym^2(C_x^*)^*$ and we show that the inclusion is strict. Let $\overline{\beta}:=-x_1^2+x_2^2+x_3^2+2x_1x_2$ and note that $\beta\in \Sym^2(C_x^*)^*$ because $\Sym^2(\phi)(\overline{\beta})=2(2z_1z_2+2z_1z_3+2z_2z_4+z_3^2+z_4^2)$ and the latter is an element of $\Sym^2(\rr_+^4)$ which consists exactly of the quadrics with nonnegative coefficients. Letting $\overline{\gamma}:= (x_1')^2-2x_1'x_2'-(x_2')^2+2(x_3')^2$ one verifies as above that $\overline{\gamma}\in \Sym^2(C_x)^*$ and that $\overline\gamma(\overline\beta)<0$ so $\overline{\beta}\not \in \Sym^2(C_x)$ as claimed.
\end{example}

\begin{cor} \label{partial} The following statements hold for marked cones,
\begin{enumerate}
\item{The tensor product (resp. symmetric power) of SDR cones is a projection of tensor products (resp. symmetric powers) of spectrahedral cones.}
\item{The tensor products (resp. symmetric powers) of duals of spectrahedra are projections of the tensor products of PSD cones $S_+(V^*)\otimes S_+(W^*)$ (resp. of symmetric powers $\Sym^n(S_+(V^*))$).}
\item{If $A^*$ and $B$ are spectrahedral cones then $\Hom_\AC(A,B)$ is isomorphic to a linear section of $\Hom_\AC(S_+(V^*),S_+(W^*))$.}
\end{enumerate}
\end{cor}

\noindent
{\bf Open Problem.} Are the Homs, tensor products and symmetric powers of spectrahedral cones SDR sets? As we will show in the next section, in general they are not  spectrahedra or even basic closed semialgebraic sets.

\section{Linearization functors on PSD cones.}\label{Sec: PSD}
Motivated by Corollary~\ref{partial} parts $(2.)$ and $(3.)$ we study the convex algebraic geometry of linearization functors on PSD cones. The results in this section suggest that with very few exceptions these are very complex and very interesting objects which have been studied in various forms in the past. 

Recall that for a vector space $V$ there is a bilinear non-degenerate pairing $\Sym^2(V^*)\times \Sym^2(V)\rightarrow \rr$ which sends $(fg,uv)$ to $\frac{f(u)g(v)+f(v)g(u)}{2}$ allowing us to identify $\Sym^2(V)$ and $\Sym^2(V^*)^*$. The following Lemma gives several interpretations for the images of linearization functors applied to PSD cones.

\begin{lemma} \label{Lemma: ID}Let $V,W$ be real vector spaces and let $n>0$ be an integer. The following canonical identifications hold,
\begin{enumerate}
\item{\label{psdd} $S_+(V)^*=S_+(V^*)$}
\item{\label{Tensor} $S_+(V)\otimes S_+(W)$ is the convex cone over the Segre-Veronese embedding of $\pp(V)\times \pp(W)$ in $\pp(\Sym^2(V)\otimes \Sym^2(W))$. $\Sym^n(S_+(V))$ is the convex cone over the $n$-th veronese re-embedding of the second veronese embedding of $\pp(V)$ in $\Sym^n(\Sym^2(V))$.}
\item{\label{tensorDual} $(S_+(V)\otimes S_+(W))^*=\Hom_\AC(S_+(V),S_+(W^*))=P_{(2,2)}$ where $P_{(2,2)}(V^*,W^*)$ is the set of nonnegative polynomials in $\Sym^2(V^*)\otimes \Sym^2(W^*)$ (i.e. of bi-degree $(2,2)$). More generally the dual of the tensor product of $n$ PSD cones is the set of nonnegative polynomials of degree $(2,\dots,2)$ in $n$ disjoint sets of variables.}
\item{ \label{symDual} $\Sym^n(S_+(V))^*$ is the set of nonnegative polynomials of degree $(2,\dots, 2)$ in $n$ sets of variables of the same size which are invariant under permutations of these sets of variables.}
\end{enumerate}
\end{lemma} 
\begin{proof} (\ref{psdd}.) The canonical identification above maps a square $u^2\in \Sym^2(V)$ to the element of $\Sym^2(V^*)^*$ given by evaluation at the point $u$. It follows that the elements of $S_+(V)^*$ are those polynomials $p\in \Sym^2(V^*)$ for which $p(u)\geq 0$ for every $u\in V$. Since every nonnegative quadric is a sum of squares this set is precisely $S_+(V^*)$ as claimed. (\ref{tensor}.) The real Segre-Veronese embedding sends $(v,w)\in \pp(V)\times \pp(W)$ to $v^2\otimes w^2\in \pp(\Sym^2(V)\otimes \Sym^2(W))$ and $S_+(V)\otimes S_+(W)$ is the cone generated by the elements of this form. The second claim is identical. (\ref{tensorDual}.) As above we identify $\Sym^2(V)\otimes \Sym^2(W)\cong \Sym^2(V^*)^*\otimes \Sym^2(W^*)^*=(\Sym^2(V^*)\otimes \Sym^2(W^*))^*$. Under this identification an element $v^2\otimes w^2$ goes to the evaluation of polynomials in $\Sym^2(V^*)\otimes \Sym^2(W^*)$ at the point $(v,w)\in V\times W$. It follows that  $(S_+(V)\otimes S_+(W))^*$ is the set of nonnegative polynomials in $\Sym^2(V^*)\otimes \Sym^2(W^*)$ as claimed. By  Lemma~\ref{Lemma: Univ} we have the equality
\[ (S_+(V)\otimes S_+(W))^*=\Hom_\AC(S_+(V)\otimes S_+(W),\rr_+)=\Hom_\AC(S_+(V),S_+(W)^*).\]
(\ref{symDual}.) By the universal property of symmetric powers we can associate to each $\phi\in \Sym^n(S_+(V))^*$ a symmetric multilinear map $T:S_+(V)^n\rightarrow \rr_+$ and in particular a multilinear map and thus, by the universal property of tensor products, an element $t\in (\bigotimes_{i=1}^n S_+(V))^*$.  By the previous paragraph $t$ is a nonnegative polynomial  in $\bigotimes _{i=1}^n\Sym^2(V^*)$ invariant under permutation of the various copies of $\Sym^2(V^*)$ (i.e. of degree $(2,\dots, 2)$, symmetric in $n$ sets of variables of the same size).
\end{proof}

\begin{theorem} Let $V,W$ be vector spaces of dimensions $m$ and $n$ respectively. The following statements hold:
\begin{enumerate}
\item{If $\min(m,n)\leq 2$ then $\Hom(S_+(V),S_+(W))$ and $\Sym^2(S_+(V))$ are SDR sets and $S_+(V)\otimes S_+(W)$ is spectrahedral.}
\item{\label{badTensor} If $m=n=3$ then $S_+(V)\otimes S_+(W)$ has a non-exposed face and in particular is not a spectrahedron.}
\end{enumerate}
\end{theorem}
\begin{proof} The map $\Psi: \Sym^2(V)\otimes \Sym^2(W)\rightarrow (\Sym^2(V^*\otimes W^*))^*$ sending an element $\phi$ to the bilinear symmetric form $(v_1^*\otimes w_1^*, v_2^*\otimes w_2^*)\rightarrow \phi(v_1^*v_2^*\otimes w_1^*w_2^*)$ satisfies $\Psi(\phi)\succeq 0$ iff $\phi$ is nonnegative on squares in $\Sym^2(V^*)\otimes \Sym^2(W^*)$. As a result the spectrahedral cone $\Psi^{-1}S_+((V^*\otimes W^*)^*)$ is dual to the cone of sums of squares in $\Sym^2(V^*)\otimes \Sym^2(W^*)$. By~\cite[Theorem 1]{Calderon} every nonnegative biquadratic form in $(m,n)$ variables is a sum of squares iff $\min(m,n)\leq 2$. Thus, if $\min(m,n)\leq 2$ then the spectrahedron $\Psi^{-1}S_+((V^*\otimes W^*)^*)$ is dual to the nonnegative polynomials in $\Sym^2(V^*)\otimes \Sym^2(W^*)$ and thus by Lemma~\ref{Lemma: ID} part (\ref{tensorDual}.) coincides with $S_+(V)\otimes S_+(W)$ as claimed. Dualizing $\Psi$ we conclude that $P_{(2,2)}$ is an SDR cone. By Lemma~\ref{Lemma: ID} part (\ref{symDual}.) $\Sym^n(S_+(V))$ is SDR since it is the dual of the intersection of an SDR set and a linear subspace. (\ref{badTensor}.) By Lemma~\ref{Lemma: ID} part (\ref{tensorDual}) we can identify $C:=S_+(V)\otimes S_+(W)$ with the space of nonnegative polynomials in $\rr[V_0,\dots,V_3,W_0,\dots, W_3]$ of degree $(2,2)$ (where $\deg(V_i)=(1,0)$ and $\deg(W_i)=(0,1)$). Via this identification the element $v^2\otimes w^2$ acts on polynomials of degree $(2,2)$ via evaluation at the point $(v,w)$. We denote this evaluation operator as $\ell_{(v,w)}$. Let 
\[ r= \left(\sum_{i=0}^4(i+1) X_i\right)^2\left(\sum_{i=0}^4(i+1) Y_i\right)^2+ \left(\sum_{i=0}^4Y_i\right)^2\left(\sum_{i=0}^4X_i\right)^2+\sum_{i=0}^3(X_iY_i)^2\]
and note that this is a nonnegative polynomial of degree $(2,2)$. The face of $C$ defined by $r$ is the cone spanned by the evaluations at the twenty real zeroes of $r$ in $\pp^3\times \pp^3$:
\begin{table}
\centering
\begin{tiny}
\begin{tabular}{|c|c|c|c|}
\hline
{[}0,0,0,1][1, -2, 1, 0] & [0, 0, 1, 0][2, -3, 0, 1] & [0, 1, 0, 0][1, 0, -3, 2] & [2, -1, 0, 0][0, 0, 1, -1]\\
{[}1,0,0,0][0, 1, -2, 1] & [0, 0, 1,-1][2, -1, 0, 0] & [0, -1, 0, 1][3, 0, -1, 0]& [0, -1, 2, -1][1, 0, 0,0]\\
{[}-1, 0, 0, 1][0, -3, 2, 0] & [0, -1, 1, 0][4, 0, 0, -1] & [-1, 0, 1, 0][0, -2, 0, 1] & [1, 0, -3, 2][0, 1, 0, 0]\\
{[}1, -1, 0, 0][0, 0, -4, 3] & [0, 0, -4, 3][1, -1, 0, 0] & [0, -2, 0, 1][1, 0, -1, 0] & [2, -3, 0, 1][0, 0, 1, 0]\\
{[}4, 0, 0, -1], [0, 1,-1, 0] & [0, -3, 2, 0][1, 0, 0, -1] & [3, 0, -1, 0][0, 1, 0, -1] & [1, -2, 1, 0][0, 0, 0, 1]\\
\hline
\end{tabular}
\end{tiny}
\end{table}

Let $Z=\{p_1,\dots, p_{20}\}$ be this set of points. We will show that the set $F={\rm Cone}(\ell_p:p\in Z\setminus \{p_1\})$ is a non-exposed face of $C$. To see that it is a face suppose $\sum \lambda_i \ell _{v_i}\in F$ with $\lambda_i>0$. It follows that $\ell_{v_i}(r)=0$ for every $i$ and thus that, up to nonnegative scaling $\ell_{v_i}= \ell_{p}$ for some $p\in Z$. Moreover $\ell_{p_1}$ cannot appear with nonzero coefficient on the left hand side since otherwise we would contradict the fact that the evaluation maps $\ell_{p_i}$ are linearly independent. To prove that $F$ is not exposed suppose the contrary. Then there exists a nonnegative polynomial $q$ of degree $(2,2)$ whose real zeroes are precisely the points of $Z\setminus \{p_1\}$. Recall that every nonnegative polynomial vanishes to order at least two at any of its zeroes and thus the existence of $q$ implies the existence of a polynomial of degree $(2,2)$ which vanishes doubly through the points of $Z\setminus \{p_1\}$ and does not vanish at $p_1$. We will show that this is impossible. To this end, let $I$ be the ideal of definition of the points in $Z$  the homogeneous coordinate ring of $\pp^3\times \pp^3$ that is,
\[ I=\sum_{i=0}^4(X_iY_i)+\left(\sum_{i=0}^4(i+1) X_i\sum_{i=0}^4(i+1) Y_i\right)+ \left(\sum_{i=0}^4Y_i \sum_{i=0}^4X_i\right)\]
and let $J$ be the ideal defined by the squares of the defining ideals of all points except the first. A calculation in the computer program Macaulay2~\cite{M2} shows that the dimensions in degree $(2,2)$ of the saturations of $J$ and $I^2$ with respect to the irrelevant ideal are both $21$ so every homogeneous polynomial of degree $(2,2)$ double-vanishing at the points of $Z\setminus \{p_1\}$ vanishes also at $p_1$, proving that $q$ does not exist. It follows that $F$ is a non-exposed face of $C$ and thus that $C$ is not a spectrahedron since every face of a spectrahedron is exposed~\cite{RG}.
\end{proof}

\begin{remark} The previous Theorem classifies all $V,W$ for which $S_+(V)\otimes S_+(W)$ is a spectrahedron. As a consequence we see that it is not true in general that the tensor product of spectrahedra is a spectrahedron. \end{remark}

Next we study the geometry of the cone $\Hom(S_+(V),S_+(W))$ following closely the ideas of Nie~\cite{Nie}. We provide a self-contained proof for the reader's benefit. Recall~\cite[Chapter I]{GKZ} that if $X\subseteq \pp(V)$ is an irreducible projective variety then a hyperplane $H\in \pp(V^*)$ is tangent to $X$ if there exists a smooth point $x\in X$ such that the tangent space to $x$ at $H$ contains the tangent space to $x$ at $X$ (equivalently such that $H$ vanishes to order at least two at $x$). Let $X^{\vee}$ denote the closure in $\pp(V^*)$ of the set of hyperplanes tangent to $X$. It is well known that $X^{\vee}$ is irreducible and in most cases a hypersurface. In this case any of its defining equations is called an $X$-discriminant. 

\begin{theorem} Let $V,W$ be vector spaces of dimensions $m,n\geq 2$. The following statements hold:
\begin{enumerate}
\item{\label{bdry} The algebraic boundary of $\Hom(S_+(V),S_+(W^*))$ (i.e. the Zariski closure of the boundary) is the irreducible hypersurface dual to $\pp(V)\times \pp(W)$ in its Segre-Veronese embedding in $\pp(\Sym^2(V)\otimes \Sym^2(W))$.}
\item{\label{intBd}The strictly positive polynomial $r:=\left(x_0^2+\dots +x_m^2\right)\left(y_0^2+\dots +y_n^2\right)$ belongs to interior of $\Hom(S_+(V),S_+(W^*))$ and to its algebraic boundary.}
\item{\label{consequences} The cone $\Hom(S_+(V),S_+(W^*))$ is not a basic closed semialgebraic set and in particular not a spectrahedron.}
\end{enumerate}
\end{theorem}
\begin{proof} (\ref{bdry}.) As above we can identify $\Hom(S_+(V),S_+(W^*))$ with the linear functionals in $\Sym^2(V^*)\otimes \Sym^2(W^*)$ which are nonnegative in the image $Z$ of $\pp(V)\times \pp(W)$ in its Segre-Veronese embedding in $\pp(\Sym^2(V)\otimes \Sym^2(W))$. Since every zero of a nonnegative polynomial in $\pp(V)\times \pp(W)$ must be a critical point (i.e. in local coordinates the gradient of the polynomial must be zero) it follows that all nonnegative polynomials with real zeroes correspond to linear forms in $Z^{\vee}$. By~\cite[Corollary 5.11]{GKZ} we know that $Z^{\vee}$ is an irreducible hypersurface. Since the boundary of  $\Hom(S_+(V),S_+(W^*))$ has real codimension one in $\Sym^2(V^*)\otimes \Sym^2(W^*)$ it follows that $Z^{\vee}$ is the algebraic closure of the boundary as claimed. (\ref{intBd}.) The polynomial obviously has no real zero in $\pp(V)\times \pp(W)$ and thus lies in the interior of the cone. In the affine chart where $x_0y_0\neq 0$ we have $\frac{\partial r}{\partial x_i}=2x_i(1+y_1^2+\dots+y_n^2)$ and $\frac{\partial r}{\partial y_i}=2y_i(1+x_1^2+\dots+x_n^2)$ and thus the polynomial vanishes to order at least two at the point $(1,i,0\dots,0)\times (1,i,0,\dots ,0)$ and thus lies in $Z^{\vee}$. (\ref{consequences}.)
If $\Hom(S_+(V),S_+(W^*))$ was defined by finitely many real polynomials $g_i\geq 0$ then one of them would have to be divisible by the $Z$-discriminant and in particular would have to vanish at an interior point of the cone which is a contradiction. 
\end{proof}

\begin{remark} Recall that a barrier function for a cone $C$ is a real-valued continuous function in the interior of $C$ which approaches $\infty$ at all points of the boundary of $C$. Log-polynomial barrier functions (i.e. those which are of the form $\log(\eta)$ for a polynomial $\eta$) are a key tool of interior point optimization algorithms~\cite{NN}. As observed by Nie in~\cite{Nie}, theorems as above imply the non-existence of a log-polynomial barrier for $\Hom(S_+(V),S_+(W^*))$. Because, if $\phi$ was a log-polynomial barrier function then $e^{\phi}$ would be an algebraic function vanishing in the boundary and thus divisible by the Z-discriminant contradicting the continuity of the barrier $\phi$ in the interior.
\end{remark}

\section{Schur functors on compact convex sets.}\label{Schur functors}
Having considered tensor product and symmetric power operations on convex sets a natural next step is to study the behavior of general Schur functors (see~\cite[Section 8.1]{F} for definitions).

Let $n, m$ be natural numbers and let $V$ be an $m$-dimensional vector space. For a partition $\lambda$ of the integer $n$ let $h_{\lambda}: V^{\oplus n}\rightarrow \mathbb{S}_{\lambda}(V)$ be the universal multilinear map which satisfies the $\lambda$-exchange axioms (see~\cite[pg. 105]{F} for a definition) sending $(v_1,\dots, v_n)$ to the class of the tableau filled with the elements $v_i$ (in order, left to right and top to bottom).
\begin{definition} Let $P\subseteq V$ be a convex body containing $0$. We define $\mathbb{S}_{\lambda}(P):={\rm Conv}\big(h_{\lambda}(P\times P\times \dots \times P)\big)\subseteq \mathbb{S}_{\lambda}(V)$. For a morphism $T:P\rightarrow Q$ of  compact convex sets let $\mathbb{S}_{\lambda}(T)$ be as in the category of vector spaces.
\end{definition}
\begin{remark} It is possible to define Schur functors in $\AC$ as we did with tensors and symmetric powers. However, if the tableaux of shape $\lambda$ has more than one row then the resulting cones are linear subspaces and thus uninteresting from the point of view of convex geometry. 
\end{remark}
Nevertheless, Schur functors are linearization functors on convex bodies (or equivalently in the category $\MC$) leading to highly symmetric convex sets. The following Theorem summarizes some of their fundamental properties,
\begin{theorem} The following statements hold
\begin{enumerate}
\item{ \label{funct} $\mathbb{S}_{\lambda}$ is a functor in the category of  compact convex sets containing the origin.}
\item{ \label{linearization} Let $W$ be any compact convex set and let $T: P ^{\oplus n}\rightarrow W$ be any multilinear map satisfying the $\lambda$-exchange axioms. There is a unique morphism of compact convex sets $t: \mathbb{S}_{\lambda}(P)\rightarrow W$ satisfying $t\circ h_{\lambda}=T$.}
\item{ \label{sim}  $\mathbb{S}_{\lambda}(P)$ is a convex body with dimension equal to the number of semistandard tableaux of shape $\lambda$ with entries in $\{1,\dots, m\}$. If $\lambda$ has more than one row then $\mathbb{S}_{\lambda}(P)$ is symmetric around the origin.}
\item{\label{extreme} If $E$ is the set of extreme points of $P$ then $\mathbb{S}_{\lambda}(P)$ equals the convex hull of $h_\lambda(E^{\oplus n})$. In particular, if $P$ is a polytope then $\mathbb{S}_{\lambda}(P)$ is a polytope.}
\item{ \label{approx} If $m_P$ is a measure supported on $P$ (for instance the restriction of the Lebesgue measure in $V$) then $m_P$ determines an accurate approximation scheme for $\mathbb{S}_{\lambda}(P)$ by projections of spectrahedra.}
\end{enumerate}
\end{theorem}
\begin{proof} (\ref{funct}.) It suffices to verify that if $f: P_1\rightarrow P_2$ is a morphism then $\mathbb{S}_{\lambda}(f)$ maps $\mathbb{S}_{\lambda}(P_1)$ to $\mathbb{S}_{\lambda}(P_2)$. Now, $\mathbb{S}_{\lambda}(f)$ maps the class of a tableau $[v_1,\dots, v_n]$ of shape $\lambda$ with entries in $P_1$ to the class of the tableau $[f(v_1),\dots, f(v_n)]$ which is an element of $\mathbb{S}_{\lambda}(P_2)$ as claimed since $f(v_i)\in P_2$ for every $i$. (\ref{linearization}.) Follows from the analogous universal property for Schur functors on vector spaces.(\ref{sim}.) The statement about dimension follows from the fact that $\mathbb{S}_{\lambda}(P)$ is full-dimensional because $P$ contains a basis for $V$. If $\lambda$ has more than one row then the set of classes of tableaux $(v_1,\dots, v_n)$ with $v_i\in P$ is closed under multiplication by $(-1)$ because the negative of a tableau can be obtained by exchanging two entries in the same column. Since this property extends to the convex closure the claim follows. (\ref{extreme}.) Since $E\subseteq P$ then $h(E^{\oplus n})\subseteq h(P^{\oplus n})$ and thus the convex hull of the left hand side is included in $\mathbb{S}_{\lambda}(P)$. By the Krein-Milman theorem every point of $P$ is a convex combination of extreme points. Since $h_\lambda$ is multilinear it follows that every element of $h_\lambda(P^{\oplus n})$ is a convex combination of elements in $h(E^{\oplus n})$ and thus equality holds. (\ref{approx}.) The product measure $m^{n}$ is supported on $P^{\oplus n}$ and the map $h_{\lambda}: V^{\oplus n}\rightarrow \mathbb{S}_{\lambda}(V)$ is a polynomial map. Lemma~\ref{Approx} gives an accurate approximation scheme for $\mathbb{S}_{\lambda}(P)$ via projections of spectrahedra.\end{proof}

\begin{example} Let $m=3$ and let $P$ be the cube with vertices $(\pm1,\pm1,\pm1)$. Let $n=2$ and let $\lambda=1+1$. Then $\mathbb{S}_{\lambda}(P)=\bigwedge^2(P)$ is the polytope shown in the figure. Note that this polytope has only $12$ vertices and thus, in contrast with what occurs with symmetric and tensor powers, not every element of $h_\lambda(E^{2})$ is an extreme point of $\mathbb{S}_{\lambda}(P)$. \begin{center}
\begin{figure}[h]
\scalebox {0.5}{
\includegraphics{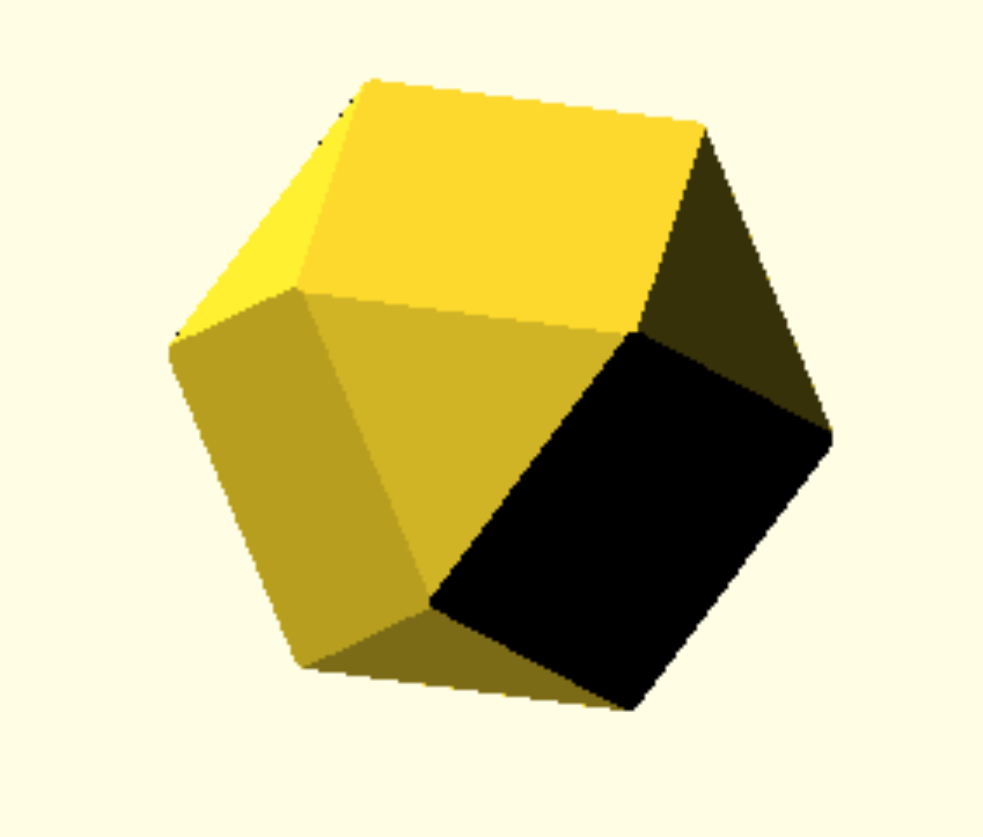}
}
\caption{The second exterior power of the cube $P$}
\end{figure}
\end{center}
\end{example}
It is an interesting problem to determine the facial structure of the compact sets obtained by applying Schur functors, even in the special case of polytopes.

\end{document}